\documentclass[a4paper,11pt,reqno]{amsart}

\usepackage[utf8x]{inputenc}
\usepackage[T1]{fontenc}
\usepackage{hyperref}
\usepackage{xcolor}
\hypersetup{
    colorlinks,
    linkcolor={red!50!black},
    citecolor={blue!50!black},
    urlcolor={blue!80!black}
}

\usepackage{microtype}
\usepackage{MnSymbol}
\usepackage{textcomp}

\usepackage[mathcal]{euscript}
\let\mathcaltmp\mathcal
\usepackage{mathrsfs}
\let\mathcal\mathscr
\let\mathscr\mathcaltmp

\makeatletter
\def\thm@space@setup{\thm@preskip=7pt
\thm@postskip=7pt}
\makeatother
\newtheoremstyle{plain}% name
  {}
  {}
  {\slshape}% Body font
  {}%Indent amount (empty = no indent, \parindent = para indent)
  {\bfseries}%  Thm head font
  {.}%       Punctuation after thm head
  { }%      Space after thm head: " " = normal interword space;
  {}%       Thm head spec

\newtheoremstyle{definition}% name
  {}
  {}
  {}
  {}
  {\bfseries}
  {.}
  { }
  {}

\makeatletter
\renewenvironment{proof}[1][\proofname]{\par
  \pushQED{\qed}%
  \normalfont \topsep0\p@\relax
  \trivlist
  \item[\hskip\labelsep\itshape
  #1\@addpunct{.}]\ignorespaces
}{%
  \popQED\endtrivlist\@endpefalse
}
\makeatother

\makeatletter
\newcommand{\eqnum}{\refstepcounter{equation}\textup{\tagform@{\theequation}}}
\makeatother

\makeatletter \@addtoreset{equation}{section} \makeatother
\renewcommand{\theequation}{\thesection.\arabic{equation}}

\newtheorem{thm}[equation]{Theorem}
\newtheorem{thmX}{Theorem}
\newtheorem{cor}[equation]{Corollary}
\newtheorem{prop}[equation]{Proposition}
\newtheorem*{defthm*}{Definition/Theorem}

\theoremstyle{definition}

\newtheorem{rem}[equation]{Remark}

\newtheorem*{exam*}{Example}

\newcommand\arXiv[1]{\href{http://arxiv.org/abs/#1}{arXiv:#1}}

\usepackage{xparse}

\usepackage{xspace}

\newcommand{\changelocaltocdepth}[1]{%
  \addtocontents{toc}{\protect\setcounter{tocdepth}{#1}}%
  \setcounter{tocdepth}{#1}}

\newcommand{\nc}{\newcommand}
\nc{\renc}{\renewcommand}
\nc{\ssec}{\subsection}
\nc{\sssec}{\subsubsection}
\nc{\on}{\operatorname}
\nc{\term}[1]{#1\xspace}

\usepackage{tikz}
\usetikzlibrary{matrix}
\usepackage{tikz-cd}

\tikzset{
  commutative diagrams/.cd,
  arrow style=tikz,
  diagrams={>=latex}}
\makeatletter
\tikzset{
  column sep/.code=\def\pgfmatrixcolumnsep{\pgf@matrix@xscale*(#1)},
  row sep/.code   =\def\pgfmatrixrowsep{\pgf@matrix@yscale*(#1)},
  matrix xscale/.code=%
    \pgfmathsetmacro\pgf@matrix@xscale{\pgf@matrix@xscale*(#1)},
  matrix yscale/.code=%
    \pgfmathsetmacro\pgf@matrix@yscale{\pgf@matrix@yscale*(#1)},
  matrix scale/.style={/tikz/matrix xscale={#1},/tikz/matrix yscale={#1}}}
\def\pgf@matrix@xscale{1}
\def\pgf@matrix@yscale{1}
\makeatother

\usepackage[inline]{enumitem}
\setlist[enumerate,1]{label={(\alph*)},itemsep=\parskip}

\newlist{thmlist}{enumerate}{1}
\setlist[thmlist,1]{
  label={\em(\roman*)}, ref={(\roman*)},
  itemsep=0.5em,
  topsep=0em,
  leftmargin=*,
  align=left,widest=vi)}

\newlist{thmlistbis}{enumerate}{1}
\setlist[thmlistbis,1]{
  label={\em(\roman*~\textit{bis})},
  ref={(\roman*}~\textit{bis}\upshape{)},
  itemsep=0.5em,
  topsep=-0.7em,
  leftmargin=0pt, align=right, widest=vi)}

\newlist{defnlist}{enumerate}{2}
\setlist[defnlist,1]{
  label={(\roman*)}, ref={(\roman*)},
  itemsep=0.5em,
  topsep=0em,
  leftmargin=*,
  align=left, widest=vi)}

\setlist[defnlist,2]{
  label={(\alph*)}, ref={(\alph*)},
  itemsep=0.75em,
  labelsep=0em,labelindent=0em,leftmargin=*,align=left,widest=vi),
  topsep=0.75em}

\newlist{defnlistbis}{enumerate}{1}
\setlist[defnlistbis,1]{
  label={(\roman*~\textit{bis})},
  ref={(\roman*}~\textit{bis}\upshape{)},
  itemsep=0.5em,
  topsep=0em,
  leftmargin=*,
  align=left, widest=vi)}

\newlist{inlinelist}{enumerate*}{1}
\setlist[inlinelist,1]{label={(\alph*)}}

\newlist{inlinedefnlist}{enumerate*}{1}
\definecolor{green}{HTML}{38550C}
\setlist[inlinedefnlist,1]{label={\color{green}(\roman*)}}

\newlist{inlinethmlist}{enumerate*}{1}
\definecolor{green}{HTML}{38550C}
\setlist[inlinethmlist,1]{label={\color{green}(\roman*)}}

\nc{\cA}{\ensuremath{\mathcal{A}}\xspace}
\nc{\cB}{\ensuremath{\mathcal{B}}\xspace}
\nc{\cC}{\ensuremath{\mathcal{C}}\xspace}
\nc{\cD}{\ensuremath{\mathcal{D}}\xspace}
\nc{\cE}{\ensuremath{\mathcal{E}}\xspace}
\nc{\cF}{\ensuremath{\mathcal{F}}\xspace}
\nc{\cG}{\ensuremath{\mathcal{G}}\xspace}
\nc{\cH}{\ensuremath{\mathcal{H}}\xspace}
\nc{\cI}{\ensuremath{\mathcal{I}}\xspace}
\nc{\cJ}{\ensuremath{\mathcal{J}}\xspace}
\nc{\cK}{\ensuremath{\mathcal{K}}\xspace}
\nc{\cL}{\ensuremath{\mathcal{L}}\xspace}
\nc{\cM}{\ensuremath{\mathcal{M}}\xspace}
\nc{\cN}{\ensuremath{\mathcal{N}}\xspace}
\nc{\cO}{\ensuremath{\mathcal{O}}\xspace}
\nc{\cP}{\ensuremath{\mathcal{P}}\xspace}
\nc{\cQ}{\ensuremath{\mathcal{Q}}\xspace}
\nc{\cR}{\ensuremath{\mathcal{R}}\xspace}
\nc{\cS}{\ensuremath{\mathcal{S}}\xspace}
\nc{\cT}{\ensuremath{\mathcal{T}}\xspace}
\nc{\cU}{\ensuremath{\mathcal{U}}\xspace}
\nc{\cV}{\ensuremath{\mathcal{V}}\xspace}
\nc{\cW}{\ensuremath{\mathcal{W}}\xspace}
\nc{\cX}{\ensuremath{\mathcal{X}}\xspace}
\nc{\cY}{\ensuremath{\mathcal{Y}}\xspace}
\nc{\cZ}{\ensuremath{\mathcal{Z}}\xspace}

\nc{\sA}{\ensuremath{\mathscr{A}}\xspace}
\nc{\sB}{\ensuremath{\mathscr{B}}\xspace}
\nc{\sC}{\ensuremath{\mathscr{C}}\xspace}
\nc{\sD}{\ensuremath{\mathscr{D}}\xspace}
\nc{\sE}{\ensuremath{\mathscr{E}}\xspace}
\nc{\sF}{\ensuremath{\mathscr{F}}\xspace}
\nc{\sG}{\ensuremath{\mathscr{G}}\xspace}
\nc{\sH}{\ensuremath{\mathscr{H}}\xspace}
\nc{\sI}{\ensuremath{\mathscr{I}}\xspace}
\nc{\sJ}{\ensuremath{\mathscr{J}}\xspace}
\nc{\sK}{\ensuremath{\mathscr{K}}\xspace}
\nc{\sL}{\ensuremath{\mathscr{L}}\xspace}
\nc{\sM}{\ensuremath{\mathscr{M}}\xspace}
\nc{\sN}{\ensuremath{\mathscr{N}}\xspace}
\nc{\sO}{\ensuremath{\mathscr{O}}\xspace}
\nc{\sP}{\ensuremath{\mathscr{P}}\xspace}
\nc{\sQ}{\ensuremath{\mathscr{Q}}\xspace}
\nc{\sR}{\ensuremath{\mathscr{R}}\xspace}
\nc{\sS}{\ensuremath{\mathscr{S}}\xspace}
\nc{\sT}{\ensuremath{\mathscr{T}}\xspace}
\nc{\sU}{\ensuremath{\mathscr{U}}\xspace}
\nc{\sV}{\ensuremath{\mathscr{V}}\xspace}
\nc{\sW}{\ensuremath{\mathscr{W}}\xspace}
\nc{\sX}{\ensuremath{\mathscr{X}}\xspace}
\nc{\sY}{\ensuremath{\mathscr{Y}}\xspace}
\nc{\sZ}{\ensuremath{\mathscr{Z}}\xspace}

\nc{\bA}{\ensuremath{\mathbf{A}}\xspace}
\nc{\bB}{\ensuremath{\mathbf{B}}\xspace}
\nc{\bC}{\ensuremath{\mathbf{C}}\xspace}
\nc{\bD}{\ensuremath{\mathbf{D}}\xspace}
\nc{\bE}{\ensuremath{\mathbf{E}}\xspace}
\nc{\bF}{\ensuremath{\mathbf{F}}\xspace}
\nc{\bG}{\ensuremath{\mathbf{G}}\xspace}
\nc{\bH}{\ensuremath{\mathbf{H}}\xspace}
\nc{\bI}{\ensuremath{\mathbf{I}}\xspace}
\nc{\bJ}{\ensuremath{\mathbf{J}}\xspace}
\nc{\bK}{\ensuremath{\mathbf{K}}\xspace}
\nc{\bL}{\ensuremath{\mathbf{L}}\xspace}
\nc{\bM}{\ensuremath{\mathbf{M}}\xspace}
\nc{\bN}{\ensuremath{\mathbf{N}}\xspace}
\nc{\bO}{\ensuremath{\mathbf{O}}\xspace}
\nc{\bP}{\ensuremath{\mathbf{P}}\xspace}
\nc{\bQ}{\ensuremath{\mathbf{Q}}\xspace}
\nc{\bR}{\ensuremath{\mathbf{R}}\xspace}
\nc{\bS}{\ensuremath{\mathbf{S}}\xspace}
\nc{\bT}{\ensuremath{\mathbf{T}}\xspace}
\nc{\bU}{\ensuremath{\mathbf{U}}\xspace}
\nc{\bV}{\ensuremath{\mathbf{V}}\xspace}
\nc{\bW}{\ensuremath{\mathbf{W}}\xspace}
\nc{\bX}{\ensuremath{\mathbf{X}}\xspace}
\nc{\bY}{\ensuremath{\mathbf{Y}}\xspace}
\nc{\bZ}{\ensuremath{\mathbf{Z}}\xspace}

\nc{\bbA}{\ensuremath{\mathbb{A}}\xspace}
\nc{\bbB}{\ensuremath{\mathbb{B}}\xspace}
\nc{\bbC}{\ensuremath{\mathbb{C}}\xspace}
\nc{\bbD}{\ensuremath{\mathbb{D}}\xspace}
\nc{\bbE}{\ensuremath{\mathbb{E}}\xspace}
\nc{\bbF}{\ensuremath{\mathbb{F}}\xspace}
\nc{\bbG}{\ensuremath{\mathbb{G}}\xspace}
\nc{\bbH}{\ensuremath{\mathbb{H}}\xspace}
\nc{\bbI}{\ensuremath{\mathbb{I}}\xspace}
\nc{\bbJ}{\ensuremath{\mathbb{J}}\xspace}
\nc{\bbK}{\ensuremath{\mathbb{K}}\xspace}
\nc{\bbL}{\ensuremath{\mathbb{L}}\xspace}
\nc{\bbM}{\ensuremath{\mathbb{M}}\xspace}
\nc{\bbN}{\ensuremath{\mathbb{N}}\xspace}
\nc{\bbO}{\ensuremath{\mathbb{O}}\xspace}
\nc{\bbP}{\ensuremath{\mathbb{P}}\xspace}
\nc{\bbQ}{\ensuremath{\mathbb{Q}}\xspace}
\nc{\bbR}{\ensuremath{\mathbb{R}}\xspace}
\nc{\bbS}{\ensuremath{\mathbb{S}}\xspace}
\nc{\bbT}{\ensuremath{\mathbb{T}}\xspace}
\nc{\bbU}{\ensuremath{\mathbb{U}}\xspace}
\nc{\bbV}{\ensuremath{\mathbb{V}}\xspace}
\nc{\bbW}{\ensuremath{\mathbb{W}}\xspace}
\nc{\bbX}{\ensuremath{\mathbb{X}}\xspace}
\nc{\bbY}{\ensuremath{\mathbb{Y}}\xspace}
\nc{\bbZ}{\ensuremath{\mathbb{Z}}\xspace}

\nc{\mrm}[1]{\ensuremath{\mathrm{#1}}\xspace}
\nc{\mit}[1]{\ensuremath{\mathit{#1}}\xspace}
\nc{\mbf}[1]{\ensuremath{\mathbf{#1}}\xspace}
\nc{\mcal}[1]{\ensuremath{\mathcal{#1}}\xspace}
\nc{\msc}[1]{\ensuremath{\mathscr{#1}}\xspace}
\nc{\mfr}[1]{\ensuremath{\mathfrak{#1}}\xspace}

\nc{\sub}{\subseteq}
\nc{\too}{\longrightarrow}
\nc{\hook}{\hookrightarrow}
\nc{\hooklongrightarrow}{\lhook\joinrel\longrightarrow}
\nc{\hooklong}{\hooklongrightarrow}
\nc{\hooklongleftarrow}{\longleftarrow\joinrel\rhook}
\nc{\twoheadlongrightarrow}{\relbar\joinrel\twoheadrightarrow}
\nc{\longrightleftarrows}{\ \raisebox{0.3ex}{\(\mathrel{\substack{\xrightarrow{\rule{1em}{0em}} \\[-1ex] \xleftarrow{\rule{1em}{0em}}}}\)}\ }

\renc{\ge}{\geqslant}
\renc{\le}{\leqslant}

\nc{\id}{\mathrm{id}}

\DeclareMathOperator{\Hom}{\on{Hom}}
\nc{\uHom}{\underline{\smash{\Hom}}}

\DeclareMathOperator{\End}{\on{End}}

\nc{\uEnd}{\underline{\smash{\End}}}

\nc{\colim}{\varinjlim}
\renc{\lim}{\varprojlim}
\nc{\Cofib}{\on{Cofib}}
\nc{\Fib}{\on{Fib}}
\nc{\initial}{\varnothing}
\nc{\op}{\mathrm{op}}

\DeclareMathOperator*{\fibprod}{\times}

\renc{\setminus}{\smallsetminus}

\usepackage{mathtools}
\newcommand{\thmref}[1]{Theorem~\ref{#1}}

\newcommand{\propref}[1]{Proposition~\ref{#1}}
\newcommand{\corref}[1]{Corollary~\ref{#1}}

\renewcommand{\eqref}[1]{(\ref{#1})}

\newcommand{\itemref}[1]{\ref{#1}}

\nc{\sing}{\mrm{sing}}
\nc{\A}{\bA}
\renc{\P}{\bP}
\nc{\V}{\bV}
\nc{\Spec}{\on{Spec}}
\nc{\D}{\on{\mbf{D}}}
\nc{\rD}{\on{\mrm{D}}}
\nc{\Dqc}{\on{\mbf{D}}_{\mrm{qc}}}
\nc{\bDelta}{\mathbf{\Delta}}
\nc{\Cech}{\textnormal{\v{C}}}
\nc{\Dperf}{\on{\mbf{D}}_{\mrm{perf}}}
\nc{\Perf}{\on{Perf}}
\nc{\Coh}{\on{Coh}}
\nc{\Qcoh}{\on{Qcoh}}
\nc{\DCoh}{\on{DCoh}}
\nc{\cl}{{\mrm{cl}}}
\nc{\Bl}{\on{Bl}}
\nc{\vir}{\mrm{vir}}
\nc{\CH}{\on{CH}}
\nc{\Zar}{\mrm{Zar}}
\nc{\et}{\mrm{\acute{e}t}}
\nc{\Nis}{\mrm{Nis}}
\renc{\H}{\on{H}}
\nc{\BM}{\mrm{BM}}
\nc{\Z}{\bZ}
\nc{\Q}{\bQ}
\nc{\K}{{\on{K}}}
\nc{\KB}{\K^{\mrm{B}}}
\nc{\Ktop}{\K^{\mrm{top}}}
\nc{\G}{{\on{G}}}
\nc{\KH}{{\on{KH}}}
\nc{\HP}{\on{HP}}
\nc{\rightrightrightarrows}
  {\mathrel{\substack{\textstyle\rightarrow\\[-0.6ex]
   \textstyle\rightarrow \\[-0.6ex]
   \textstyle\rightarrow}}}
\nc{\rightrightrightrightarrows}
  {\mathrel{\substack{\textstyle\rightarrow\\[-0.6ex]
   \textstyle\rightarrow \\[-0.6ex]
   \textstyle\rightarrow \\[-0.6ex]
   \textstyle\rightarrow}}}
\nc{\Einfty}{{\sE_\infty}}
\renc{\sp}{\mrm{sp}}
\nc{\Td}{\on{Td}}
\nc{\ch}{\on{ch}}
\nc{\RGamma}{R\Gamma}
\nc{\red}{\mrm{red}}
\nc{\der}{{\mrm{der}}}
\nc{\Mod}{{\mrm{Mod}}}
\nc{\Gr}{{\on{Gr}}}
\nc{\Ind}{\on{Ind}}
\nc{\Pro}{\on{Pro}}
\nc{\dash}{{\textnormal{-}}}
\nc{\InftyCat}{\infty\dash\mrm{Cat}}
\nc{\Pres}{\mrm{Pres}}
\nc{\form}{\widehat}
\nc{\R}{\bR}
\renc{\L}{\bL}
\nc{\otimesL}{\mathchoice{\overset{\bL}{\otimes}}{\otimes^\bL}{\otimes^\bL}{\otimes^\bL}}
\nc{\fibprodR}{\fibprod^\bR}
\nc{\uRHom}{\bR\uHom}
\nc{\GL}{\mrm{GL}}
\nc{\SW}{\on{SW}}
\nc{\Vect}{\on{Vect}}
\nc{\Fun}{\on{Fun}}
\nc{\Nat}{\on{Nat}}
\nc{\un}{\mbf{1}}
\nc{\pr}{\mrm{pr}}
\nc{\pt}{\mrm{pt}}
\nc{\vb}[1]{\langle{#1}\rangle}
\nc{\Pt}{\on{Pt}}
\nc{\lisse}{{\triangleleft}}
\nc{\Lis}{\mrm{Lis}}
\nc{\LisStk}{\mrm{LisStk}}
\nc{\Et}{{\mrm{Et}}}
\nc{\aff}{\mrm{aff}}
\nc{\qproj}{\mrm{qproj}}
\nc{\fp}{\mrm{fp}}
\nc{\ft}{\mrm{ft}}
\nc{\affft}{\mrm{affft}}
\nc{\sm}{\mrm{sm}}
\nc{\lci}{\mrm{lci}}
\nc{\Lisftsm}{\Lis^{\ft:\sm}}
\nc{\Lisaffftsm}{\Lis^{\affft:\sm}}
\nc{\Lisaspftsm}{\Lis^{\mrm{aspft}:\sm}}
\renc{\top}{\mrm{top}}
\nc{\C}{\on{C}}
\nc{\Chom}{\mrm{C}_\bullet}
\nc{\Ccoh}{\mrm{C}^\bullet}
\nc{\Ccohc}{\mrm{C}_{\mrm{c}}^\bullet}
\nc{\CBM}{\mrm{C}^{\BM}_\bullet}
\nc{\mot}{\mrm{mot}}
\nc{\Chommot}{\mrm{C}^{\mot}_\bullet}
\nc{\Top}{\mrm{Top}}
\renc{\top}{\mrm{top}}
\nc{\Spc}{\mrm{Spc}}
\nc{\Stk}{\mrm{Stk}}
\nc{\Art}{\mrm{Art}}
\nc{\Shv}{\on{Shv}}
\nc{\Spt}{\mrm{Spt}}
\nc{\heart}{\heartsuit}
\nc{\an}{\mrm{an}}
\nc{\Anima}{\mrm{Ani}}
\nc{\Aff}{\mrm{Aff}}
\nc{\MotSpc}{{\mrm{MAni}}}
\nc{\SH}{\on{\mathbf{SH}}}
\renc{\L}{\mrm{\bL}}
\nc{\h}{\mrm{h}}
\nc{\Sm}{\mrm{Sm}}
\nc{\Sch}{\mrm{Sch}}
\nc{\Asp}{\mrm{Asp}}
\nc{\Betti}{\mrm{Bet}}
\nc{\cdh}{\mrm{cdh}}
\renc{\Re}{\mrm{Re}}
\nc{\bz}{\mathbf{z}}
\nc{\Tot}{\on{Tot}}
\nc{\MGL}{\mrm{MGL}}
\nc{\modmod}{/\!\!/}
\nc{\dR}{\mrm{dR}}
\nc{\laur}{(\!\!(u)\!\!)}
\nc{\IndCoh}{\on{IndCoh}}

\nc{\scr}{\term{derived commutative ring}}
\nc{\scrs}{\term{derived commutative rings}}

\nc{\inftyCat}{\term{$\infty$-category}}
\nc{\inftyCats}{\term{$\infty$-categories}}

\nc{\inftyGrpd}{\term{$\infty$-groupoid}}
\nc{\inftyGrpds}{\term{$\infty$-groupoids}}

\nc{\dA}{\term{derived Artin}}

\title{The lattice property for perfect complexes on singular stacks\vspace{-2mm}}
\author[A.\,A. Khan]{Adeel A. Khan}%\\
\date{2023-08-03}

\makeatletter
\def\l@subsection{\@tocline{2}{0pt}{4pc}{6pc}{}}
\makeatother

\begin{document}

\begin{abstract}
  Let $\sC$ be the stable \inftyCat of perfect complexes on a derived Deligne--Mumford stack $\sX$ of finite type over the complex numbers.
  We prove that the complexified noncommutative topological Chern character $\Ktop(\sC) \otimes \bC \to \HP(\sC)$ is invertible.
  In the appendix we show the same property for $\sC$ the stable \inftyCat of coherent complexes on a derived algebraic space.
  \vspace{-5mm}
\end{abstract}

\maketitle

% \renewcommand\contentsname{\vspace{-1cm}}
% \tableofcontents

\setlength{\parindent}{0em}
\parskip 0.6em

\thispagestyle{empty}

%%%%%%%%%%%%%%%%%%%%%%%%%%%%%%%%%%%%%%%%%%%%%%%%%%%%%%%%%%%%%%%%%%%%%%%%%%%

\changelocaltocdepth{1}

Let $\sC$ be a $\bC$-linear stable \inftyCat (or pretriangulated dg-category).
The topological K-theory of $\sC$ in the sense of Blanc \cite{Blanc} is a spectrum $\Ktop(\sC)$ which admits a canonical map to the periodic cyclic homology spectrum
\begin{equation*}
  \Ktop(\sC) \to \HP(\sC)
\end{equation*}
that may be regarded as a ``noncommutative'' analogue of the Chern character.
Let us say that $\sC$ satisfies the \emph{lattice property} when the induced map $\Ktop(\sC) \otimes \bC \to \HP(\sC)$ is invertible.

The \emph{lattice conjecture}, motivated by considerations in noncommutative Hodge theory, is the assertion that any \emph{smooth} and \emph{proper} $\sC$ satisfies the lattice property (see \cite[2.2.6(b)]{KatzarkovKontsevichPantev}, \cite[Conj.~1.7]{Blanc}).
For $\sC$ the stable \inftyCat of perfect complexes on a scheme or stack, smoothness and properness do not appear to be relevant.
For example, the lattice property is known for the stable \inftyCat $\Perf(X)$ of perfect complexes on any quasi-separated derived algebraic space $X$ of finite type over $\bC$ (see \cite[Prop.~4.32]{Blanc}, \cite[Cor.~6.8]{Konovalov}).
Halpern--Leistner and Pomerleano extended this to \emph{smooth} Deligne--Mumford stacks as well as certain \emph{smooth} global quotient stacks (see \cite[Thm.~2.17, Cor.~2.19]{HalpernLeistnerPomerleano}).
In this paper we consider the singular case:

\begin{thmX}\label{thm:perf}
  For a derived stack $\sX$ of finite type over $\bC$, the $\bC$-linear stable \inftyCat $\Perf(\sX)$ satisfies the lattice property in the following cases:
  \begin{thmlist}
    \item\label{item:perf/dm}
    $\sX$ is Deligne--Mumford with separated diagonal.

    \item\label{item:perf/quot}
    $\sX = [X/G]$ where $X$ is a quasi-separated derived algebraic space of finite type and $G$ is an affine algebraic group with diagonalizable identity component.
  \end{thmlist}
\end{thmX}

In fact we prove the result more generally for derived Artin stacks $\sX$ with separated diagonal whose stabilizers are \emph{nice} algebraic groups in the sense of \cite[Def.~1.1]{HallRydhGroups}.
The main new tool is an equivariant cdh descent result for truncating invariants of stable \inftyCats.
Case~\itemref{item:perf/quot} of \thmref{thm:perf} was conjectured by Halpern--Leistner and Pomerleano (without the diagonalizability hypothesis).

In the appendix we record a proof of the following result.
We write $\DCoh(X)$ for the stable \inftyCat of coherent (= pseudocoherent with bounded cohomology) complexes on a derived algebraic space $X$.

\begin{thmX}\label{thm:coh}
  Let $X$ be a quasi-separated derived algebraic space of finite type over $\bC$.
  The $\bC$-linear stable \inftyCat $\DCoh(X)$ satisfies the lattice property.
\end{thmX}

Recently, Brown and Walker \cite{BrownWalker} proved this for $X$ a local complete intersection scheme using a dévissage result for $\HP$.
Our main ingredient is a stronger dévissage result that follows from work of Preygel \cite{Preygel}.

\ssec*{Acknowledgments}

I would like to thank Harrison Chen, Benjamin Gammage, Daniel Halpern-Leistner, Mark E. Walker, and Charanya Ravi for helpful discussions and comments on previous drafts.
I was supported by the grants NSTC 110-2115-M-001-016-MY3 and AS-CDA-112-M01 (Academia Sinica).

\section{Cdh descent}

  Let $G$ be an fppf group scheme over an affine scheme $S$, which we assume noetherian and of finite Krull dimension for simplicity.
  We will assume that $G$ is \emph{embeddable}, i.e., can be embedded as a closed subgroup of $\GL_{n,S}$ for some $n$, and that $G$ is \emph{nice}, i.e., an extension of a tame finite étale group scheme by a group scheme of multiplicative type (see \cite[Def.~2.1]{AlperHallRydh}).
  In particular, $G$ is linearly reductive.

  Let $E$ be a localizing invariant of $\sO_S$-linear stable \inftyCats\footnote{%
     with values in spectra, say, or any stable \inftyCat with an exact conservative functor to spectra
  }.
  On the \inftyCat $\sS^G$ of quasi-separated derived algebraic spaces of finite type over $S$ with $G$-action, the presheaf
  $$E^G(-) := E(\Perf([-/G])$$
  satisfies Nisnevich descent by \cite[Thm.~1.40, Rem.~2.15]{kstack}.
  By the generalized Sumihiro theorem (see \cite[Thm.~2.14(ii)]{sixstack}, \cite[Prop.~A.1.9]{BKRSMilnor}), every $X \in \sS^G$ admits a $G$-equivariant scallop decomposition by quasi-affines.
  If $X \in \sS^G$ is quasi-affine it moreover admits a $G$-equivariant scallop decomposition by affines (see \cite[Prop.~A.1.9]{BKRSMilnor}).
  We will use these observations repeatedly in combination with Nisnevich descent to reduce statements about $X \in \sS^G$ to the affine case.

  An \emph{abstract blow-up square} in $\sS^G$ is a commutative square
  \begin{equation}\label{eq:mustafina}
    \begin{tikzcd}
      Z' \ar{r}\ar{d}
      & X' \ar{d}{f}
      \\
      Z \ar{r}{i}
      & X
    \end{tikzcd}
  \end{equation}
  which is cartesian on classical truncations, where $i$ is a closed immersion and $f$ is a proper morphism inducing an isomorphism $X'\setminus f^{-1}(Z) \simeq X\setminus Z$.
  The \emph{cdh} topology on $\sS^G$ is generated by Nisnevich covers and, for every abstract blow-up square as above, the cover $Z \coprod X' \to X$.
  The following cdh descent criterion is from \cite[Thm.~5.6]{kblow}:

  \begin{thm}
    The presheaf $X \mapsto E^G(X)$ satisfies cdh hyperdescent on $\sS^G$ if and only if it satisfies nil-invariance: for every $X \in \sS^G$ and every surjective closed immersion $i : Z \to X$, the induced map
    \[ i^* : E^G(X) \to E^G(Z) \]
    is invertible.
  \end{thm}
  \begin{proof}
    For any $X \in \sS^G$ the inclusion $\initial \to X$ is a proper morphism inducing an isomorphism over $X\setminus Z = \initial$, so the condition is necessary.
    For the other direction, we first apply nil-invariance for the inclusion of the classical truncation to restrict our attention to abstract blow-up squares over $X \in \sS^G$ classical.
    By \cite[Thm.~C, Rem.~0.0.8]{BKRSMilnor}, we have \emph{pro}-excision for any such blow-up square.
    By nil-invariance applied to infinitesimal thickening, this reduces to ordinary excision.
  \end{proof}

  We deduce a $G$-equivariant version of \cite[Thm.~E]{LandTamme}.

  \begin{cor}\label{cor:cdh}
    If $E$ is a truncating invariant in the sense of \cite[Def.~3.1]{LandTamme}, then $E^G(-)$ satisfies cdh hyperdescent on $\sS^G$.
  \end{cor}
  \begin{proof}
    Since $E$ is truncating, we have nil-invariance by \cite{ElmantoSosnilo} (combining Thm.~1.0.4 and Prop.~5.1.10).
  \end{proof}

\section{Proof of \thmref{thm:perf}}

  Denote by $F(-)$ the fibre of the natural transformation $\Ktop(-) \otimes \bC \to \HP(-)$.
  This is a localizing invariant of $\bC$-linear stable \inftyCats, which is truncating by (the proof of) \cite[Cor.~5.6]{Konovalov}.
  We will prove the following, which generalizes both cases of \thmref{thm:perf}:

  \begin{thm}
    Let $\sX$ be a derived algebraic stack of finite type over $\bC$ with separated diagonal and nice stabilizers.
    Then $\Perf(\sX)$ satisfies the lattice property, i.e., $F(\sX) \simeq 0$.
  \end{thm}

  \begin{proof}
    By \cite[Thm.~2.12(ii)]{sixstack} (based on \cite[Thm.~1.9]{AlperHallHalpernLeistnerRydh}), $\sX$ is nicely scalloped; that is, it admits a scallop decomposition by quotient stacks $[X/G]$ with $G$ a nice embeddable group scheme over an affine $\bC$-scheme $S$, acting on a finite type quasi-affine derived scheme $X$ over $S$.
    By Nisnevich descent it will thus be enough to show that $F^G(X) := F([X/G]) \simeq 0$ with $X$ and $G$ as above.
    Repeating the same reasoning with \cite[Prop.~A.1.9]{BKRSMilnor}, we may moreover assume that $X$ is affine.

    Since $F$ is truncating, $F^G(-)$ satisfies cdh hyperdescent and nil-invariance by \corref{cor:cdh}.
    In particular, we may assume that $X$ is classical and reduced.
    Since $X$ admits a $G$-equivariant resolution of singularities (e.g. by \cite[Thm.~8.1.1]{AbramovichTemkinWlodarczyk} applied to $[X/G]$), there exists a $G$-equivariant cdh hypercover $\widetilde{X}_\bullet \to X$ where each $\widetilde{X}_n$ is smooth.
    By cdh hyperdescent again, we may thus assume that $X$ is smooth.
    Note that $X$ need no longer be affine, but since $G$ is nice we may apply generalized Sumihiro and Nisnevich descent to assume $X$ affine again.

    Thus suppose $X$ is smooth and affine.
    In case $G$ is defined over $\bC$ (e.g. $G=\GL_{n,S}$), the claim is a special case of \cite[Thm.~2.17]{HalpernLeistnerPomerleano}, where $[X/G]$ admits a ``semicomplete KN stratification'' by \cite[Thm.~1.3]{HalpernLeistnerPomerleano} because $G$ is reductive and $X$ is affine.

    Otherwise, choose an embedding $G \sub \GL_{n,S}$ and write
    $$[X/G] \simeq [(X \fibprod_S^G \GL_{n,S}) / \GL_{n,S}],$$
    where $X \fibprod_S^G \GL_{n,S} = [(X \fibprod_S \GL_{n,S})/G]$, with $G$ acting on $X \times \GL_{n,S}$ by $h \cdot (x, g) = (h \cdot x, g\cdot h^{-1})$ and $\GL_{n,S}$ acting on $X \times \GL_{n,S}$ by $h \cdot (x, g) = (x, h \cdot g)$ (this passes to $X \fibprod^G \GL_{n,S}$ since the actions commute).
    Since $G$ is linearly reductive, $G/\GL_{n,S}$ is affine by Matsushima, so $X \fibprod^G \GL_{n,S}$ is affine (and still smooth).
    Thus $$F^G(X) \simeq F^{\GL_n}(X \fibprod^G \GL_n) \simeq 0$$
    as desired.
  \end{proof}

\appendix
\section{The lattice property for \texorpdfstring{$\DCoh(X)$}{DCoh(X)}}

  \subsection{Dévissage for periodic cyclic cohomology}

    Given a localizing invariant $E$, we write $E^\BM(-) := E(\DCoh(-))$.
    For any closed immersion of (derived) algebraic spaces $i : Z \to X$ there is a canonical map
    \begin{equation}\label{eq:wildcat}
      E^\BM(Z) = E(\DCoh(Z)) \to E(\DCoh(X~\mrm{on}~Z))
    \end{equation}
    where $\DCoh(X~\mrm{on}~Z)$ is the kernel of the restriction functor $\DCoh(X) \to \DCoh(X\setminus Z)$.
    Since $E$ is localizing, the target is identified with
    \[
      E(\DCoh(X~\mrm{on}~Z))
      \simeq \Fib(E^\BM(X) \to E^\BM(X\setminus Z)).
    \]
    Thus \eqref{eq:wildcat} is invertible if and only if the sequence (which is canonically null-homotopic)
    \[
      E^\BM(Z) \to E^\BM(X) \to E^\BM(X\setminus Z)
    \]
    is exact.

    For algebraic K-theory, hence also for $\Ktop$, Quillen's dévissage theorem implies \eqref{eq:wildcat} is invertible.
    This is not the case for arbitrary localizing invariants (see \cite[Ex.~1.11]{Keller} for a counterexample in $E=\on{HH}$).

    \begin{thm}\label{thm:dev}
      Let $X$ be an algebraic space of finite type over $\bC$.
      For any closed immersion $i : Z \hook X$, the canonical map $\HP^\BM(Z) \to \HP(\DCoh(X~\mrm{on}~Z))$ is invertible.
    \end{thm}

    \begin{rem}
      When $X$ is smooth and $i : Z \hook X$ is a a quasi-smooth closed immersion (with $Z$ possibly derived), Brown and Walker recently gave a different proof of \thmref{thm:dev}.
      Indeed, using \propref{prop:ebmnis} below and the local structure of quasi-smooth closed immersions (see \cite[2.3.6]{blowups}), one reduces to the local calculation of \cite[Thm.~4.2(i)]{BrownWalker}.
      If we know that $\HP^\BM$ is insensitive to derived structures, this gives another proof of \thmref{thm:dev} for $X$ smooth (and $Z$ any closed subspace), because every closed subspace of $X$ admits some quasi-smooth derived structure locally on $X$.
      In fact, $\HP^\BM$ is indeed insensitive to derived structures if we admit \thmref{thm:prey} below (in view of the localization triangle \eqref{eq:trifling} for the closed immersion $X_\cl \to X$), but we do not know a direct proof that does not go through Preygel's comparison.
    \end{rem}

    The following is a consequence of dévissage (as in \cite[Cor.~3.11]{kstack}), but in fact holds more generally:

    \begin{prop}\label{prop:ebmnis}
      Let $E$ be a localizing invariant of $k$-linear stable \inftyCats (for a commutative ring $k$).
      Then $E^\BM(-)$ satisfies Nisnevich descent on qcqs derived algebraic spaces over $k$.
    \end{prop}
    \begin{proof}
      Let $X$ be a qcqs derived algebraic space.
      For every étale $U$ over $X$, there is a canonical equivalence
      \[ \Perf(U) \otimes_{\Perf(X)} \DCoh(X) \simeq \DCoh(U) \]
      by \cite[Chap.~4, Rem.~3.3.3]{GaitsgoryRozenblyum}.
      Consider then the localizing invariant $E'$ of $\Perf(X)$-linear stable \inftyCats\footnote{%
        See \cite[App.~A]{ClausenMathewNaumannNoel} for this notion.
      } given by
      \[E'(\sC) := E(\sC \otimes_{\Perf(X)} \DCoh(X)),\]
      so that $E'(\Perf(-)) \simeq E^\BM(-)$ on the small étale site of $X$.
      By \cite[Prop.~A.15]{ClausenMathewNaumannNoel}, $E'(\Perf(-))$ satisfies Nisnevich descent, hence so does $E^\BM(-)$.
    \end{proof}

    The following is \cite[Thm.~1.1.2, Thm.~6.3.2]{Preygel}:

    \begin{thm}[Preygel]\label{thm:prey}
      Let $X$ be a quasi-separated derived algebraic space locally of finite type over $\bC$.
      Then there is a canonical isomorphism
      \[
        \HP^\BM(X)
        \to \Chom^{\BM,\dR}(X) \otimes_k k\laur
      \]
      where $u$ is in homological degree $-2$.
      Moreover, it is covariantly functorial with respect to proper push-forwards and contravariantly functorial with respect to quasi-smooth pull-backs.
    \end{thm}

    This immediately implies \thmref{thm:dev}.
    To see this we recall the definition of the complex of de Rham Borel--Moore chains on $X$, for $X$ locally of finite type over a field $k$ of characteristic zero:
    \[
      \Chom^{\BM,\dR}(X) := \RGamma(X, \omega^\dR_X)
    \]
    where $\omega^\dR_X$ denotes the dualizing complex of $X$ in the \inftyCat of $D$-modules.
    Following \cite{GaitsgoryRozenblyum}, the latter is by definition the \inftyCat $\IndCoh(X_\dR)$ of ind-coherent sheaves on the \emph{de Rham prestack} $X_\dR$ of $X$ (see \cite[1.2.1]{Preygel}, \cite[\S 1]{GaitsgoryRozenblyum}).
    The de Rham dualizing complex $\omega^\dR_X$ is just the dualizing complex of $X_\dR$ in $\IndCoh(X_\dR)$.
    Thus more explicitly,
    \[
      \Chom^{\BM,\dR}(X) = \RGamma(X_{\dR}, \omega^{\DCoh}_{X_\dR}).
    \]

    Let $i : Z \to X$ be a closed immersion and $j$ the inclusion of the complement $X\setminus Z$.
    Kashiwara's lemma \cite[Prop.~2.5.6]{GaitsgoryRozenblyum} implies that we have an exact triangle of functors
    \[ i_*i^! \to \id \to j_*j^! \]
    where the functoriality is at the level of $D$-modules.
    This gives rise to the localization exact triangle
    \begin{equation}\label{eq:trifling}
      \Chom^{\BM,\dR}(Z)
      \to \Chom^{\BM,\dR}(X)
      \to \Chom^{\BM,\dR}(X\setminus Z).
    \end{equation}

    \begin{proof}[Proof of \thmref{thm:dev}]
      By \thmref{thm:prey}, the sequence
      \[
        \HP^\BM(Z)
        \to \HP^\BM(X)
        \to \HP^\BM(X\setminus Z)
      \]
      is identified with $\textrm{\eqref{eq:trifling}} \otimes k\laur$.
    \end{proof}

  \subsection{Proof of \thmref{thm:coh}}

    Let $X$ be a quasi-separated derived algebraic space of finite type over $\bC$.
    Let $F(-)$ denote the fibre of $\Ktop(-) \otimes \bC \to \HP(-)$, regarded as a localizing invariant of $\bC$-linear stable \inftyCats.
    We will show that $F^\BM(X) := F(\DCoh(X)) \simeq 0$.

    By \propref{prop:ebmnis} the claim is Nisnevich-local on $X$, so we may assume that $X$ is an affine scheme.
    In particular, there exists a closed immersion $X \hook Y$ where $Y$ is a smooth affine $\bC$-scheme.
    Since both $\K^{\mrm{top,BM}}$ and $\HP^\BM$ satisfy dévissage (the former follows from the case of $\K^\BM=\G$ and the latter is \thmref{thm:dev}), so does $F^\BM$.
    That is, we have an exact triangle
    \[
      F^\BM(X) \to F^\BM(Y) \to F^\BM(Y\setminus X).
    \]
    Since $Y$ and $Y\setminus X$ are regular, we have $F^\BM(Y) \simeq F(\Perf(Y)) \simeq 0$ and $F^\BM(Y\setminus X) \simeq 0$ by \cite[Prop.~4.32]{Blanc}.
    The claim follows.

%%%%%%%%%%%%%%%%%%%%%%%%%%%%%%%%%%%%%%%%%%%%%%%%%%%%%%%%%%%%%%%%%%%%%%%%%%%

%!TEX root = lettuce.tex

\bibliographystyle{halphanum}

Institute of Mathematics, Academia Sinica, 10617 Taipei, Taiwan

\end{document}